\documentclass[a4paper,12pt,wlayout]{article}

\usepackage{array} 
\usepackage{amssymb}
\usepackage{amsmath} 
\usepackage{amsbsy} 
\usepackage{epsfig}
\usepackage{mathrsfs} 
\usepackage{graphics}
\usepackage{graphicx}
\usepackage[sort&compress,numbers]{natbib}
\usepackage{pstricks}
\usepackage{float}
\usepackage{amsthm}
\usepackage{slashed}
\usepackage{a4wide}

\newcommand{\eps}{\varepsilon}
\newcommand{\lb}{\label}
\newcommand{\go}{\rightarrow}

\newcommand{\ee}{\end{equation}}
\newcommand{\be}{\begin{equation}}
\newcommand{\bea}{\begin{eqnarray}}
\newcommand{\eea}{\end{eqnarray}}
\newcommand{\sbea}{\begin{subequations}\begin{eqnarray}}
\newcommand{\seea}{\end{eqnarray}\end{subequations}} 
\newcommand{\ees}{\end{equation*}}
\newcommand{\bes}{\begin{equation*}}
\newcommand{\beas}{\begin{eqnarray*}}
\newcommand{\eeas}{\end{eqnarray*}}

\newcommand{\rf}[1]{(\ref{#1})}

\newcommand{\const}{\mathrm{const}}
\newcommand{\mR}{\mathbb{R}\,}

\theoremstyle{plain}
\newtheorem{theorem}{\bf Theorem}[section]
\newtheorem{lemma}[theorem]{\bf Lemma}

\theoremstyle{remark}
\newtheorem{remark}[theorem]{\bf Remark}

\begin{document}
\title{Asymptotic decay and non-rupture of viscous sheets}

\author{M. Fontelos$^1$, G. Kitavtsev$^2$ \corresp{\email{georgy.kitavtsev@bristol.ac.uk}} and R. Taranets$^3$}

\author{Marco Fontelos\thanks{Instituto de Ciencias Matem\'aticas,
 (ICMAT, CSIC-UAM-UCM-UC3M), C/ Serrano 123, 28006 
Madrid, Spain.}, Georgy Kitavtsev\thanks{School of Mathematics, University of Bristol, University Walk,
Bristol, BS81TW, UK} \hspace{-1ex} and Roman M. Taranets\thanks{Institute of Applied Mathematics and Mechanics of the NASU, Dobrovolskogo Str. 1, 84100, Sloviansk, Ukraine}}

\maketitle

\begin{abstract}
For a nonlinear system of coupled PDEs, that describes evolution of a viscous thin liquid sheet and takes account of surface tension at the free surface, we show exponential $(H^1,\,L^2)$ asymptotic decay to the flat profile of its solutions considered with general initial data.
Additionally, by transforming the system to Lagrangian coordinates we show
that the minimal thickness of the sheet stays positive for all times. This result
proves the conjecture formally accepted in the physical literature~\cite{EF15}, that a
viscous sheet can not rupture in finite time in absence of external forcing. Moreover, in absence of surface tension we find a special
class of initial data for which the Lagrangian solution exhibits $L^2$-exponential decay to the flat profile.
\end{abstract}

\section{Introduction}\label{A}
The last decades showed a considerable progress in mathematical understanding of
singularity formation during dewetting of thin viscous liquid films, see e.g. reviews~\cite{ODB97,EV08,CM09}.
In particular, topics of finite time rupture of thin films and pinch-off of a liquid thread
driven by attractive intermolecular van der Waals forces and curvature, respectively, were intensively studied
both analytically and numerically starting from the pioneering articles~\cite{WD82,ED93,E93}.
Fundamental forces in these processes include van der Waals, surface tension, viscous and inertia ones~\cite{kad94,mik96,lis99,WB99,WB00,vay01,VLW01}.
Let us also mention recent articles~\cite{BT13,KFE17} of the infinite time rupture in viscous sheets
driven by Marangoni forces in presence of gradients of temperature or surfactant distributions at the sheet free surface.

In the modelling of dewetting processes in micro- and nanoscopic liquid films lubrication approximation resulting in high-order degenerate
parabolic equations occurs to be especially effective, see e.g.~\cite{ODB97,MWW06} and references therein.
In this article, we study asymptotic  behaviour of solutions to the system describing evolution of a thin viscous sheet~\cite{EF15}:
\begin{subequations}
\label{SSM}
\begin{align}
v_t+vv_x&=\sigma\,h_{xxx}+\frac{\nu}{h}[hv_x]_x,
\label{SSM1}\\
h_t&= -\left(h v\right)_x.
\label{SSM2}
\end{align}
\end{subequations}
Here, $v(x,t)$ and $h(x,t)$ denote the average velocity in the lateral
direction and the height profile for the free surface, while $\sigma,\,\nu$
are dimensionless surface tension and viscosity, respectively.
The high order of system \rf{SSM} is a result of the contribution from surface tension at the free boundary, reflected by the linearised curvature term $\sigma h_{xx}$. 
The terms $v_t+vv_x$ and $\nu(hu_x)_x/h$ in \rf{SSM1} represent inertial and Trouton viscosity terms, respectively.

Existence of weak solutions to \rf{SSM} considered in bounded domains
$Q_T=(0,T)\times\Omega$, where $\Omega=(0,1)$, with zero velocity and homogeneous
Neumann boundary conditions for $h$:
\be
h_x(0,t)=h_x(1,t)=v(0,t)=v(1,t)=0
\lb{BC}
\ee
and initial data with positive height:
\bes
h(x,0)=h_0(x)>0,\  v(x,0)=v_0(x)\quad\text{for}\quad x\in[0,\,1],
\ees
was shown by~\citet{KLN11} via an additional introduction of the regularising Lennard-Jones potential in \rf{SSM1}.
Observe, that the boundary conditions \rf{BC} for $v$ guarantee the mass conservation
\be
\int_0^1 h(x,t)\,dx =||h_0||_1:=M>0\quad \text{for all}\  t\geqslant  0.
\lb{CM}
\ee

To summarize what is known or widely accepted about pinch-off singularities of system \eqref{SSM}.
Finite time rupture of solutions of \rf{SSM} under additional presence of van der Waals forces in \rf{SSM1}
was investigated recently both numerically and analytically in~\cite{VLW01,PMN10}.
They were concerned with existence and stability of the self-similar solutions dynamically describing the neighborhood of the pinch-off.  In~\cite{VLW01} it was shown that the rupture occurs due to combined
competition between inertia, viscosity and van der Waals terms, while
surface tension staying negligible. The  self-similar solutions in this case
exhibit the rupture in finite time with the minimum height evolving as
\bes
\min_{x\in\Omega} h(x,t) \sim (T^*-t)^{1/3},
\ees
where $T^*$ is the rupture time.
In turn, infinite time rupture of solutions to \rf{SSM} considered without van
der Waals forces but rather coupled through the Marangoni term with the
diffusion equation for surfactant distribution at the sheet free surface was described recently in~\cite{BT13,KFE17}. In~\cite{KFE17} it was shown that the minimum height in this case follows the self-similar law
\bes
\min_{x\in\Omega} h(x,t) \sim \exp\{-a\nu t\},
\ees
where $a>0$ is the thinning rate of the sheet depending both on the initial
data and physical parameters of the system.

Finally, for system \rf{SSM} in absence of any additional forces and in inviscid case, i.e. $\nu=0$, the finite time rupture can occur for suitable initial conditions~\cite{Ma76,PS98,BT07} .
The neighborhood of the pinch point is described then by the similarity solution of~\cite{BT07} with
\be
\min_{x\in\Omega} h(x,t) \sim (T^*-t)^{\beta},\quad \max_{x\in\Omega} v(x,t) \sim (T^*-t)^{-\gamma}\quad\text{with}\ \beta=0.7477,\ \gamma=0.3131,
\lb{BT_ss}
\ee
and the rupture is driven by the competition of inertia and curvature terms in \rf{SSM}.

In contrast to the mentioned studies, the main focus of this article is to consider system \rf{SSM} with $\nu>0$  and to prove that finite time rupture is not possible in this case, but rather the exponential asymptotic decay to the flat profile $(h,\,v)=(M,0)$, where $M$ is from \rf{CM}, occurs.
The fact, that viscous sheets can not rupture in absence of van der Waals or external forcing is formally well-accepted in physical literature (cf.~\cite{BT13, KFE17} and references therein). In fact, in the book~\cite{EF15} it was pointed out that system \rf{SSM} with $\sigma>0$ and $\nu>0$
does not admit consistent self-similar scalings. In particular, if one
considers the self-similar scaling \rf{BT_ss} in the case $\nu>0$,  one finds
out that the viscous Trouton term enters the leading order balance and becomes
dominant, by that  breaking the self-similar ansatz \rf{BT_ss}, when the minimum height decreases to {\it the critical range}:
\be
h_{min}(t)\gtrsim\nu^2,
\lb{LB}
\ee
Indeed, as it shown in Example 7.7 of~\cite{EF15} the surface tension and viscous terms
in \rf{SSM1} scale like $\sigma\tau^{\beta-2}$
and $\nu\tau^{-\beta-1}$, respectively, where $\tau^\beta$ is the self-similar spatial
scaling for the inviscid pinch-off analysed by~\cite{BT07}. Balancing these
two terms and using the fact $h_{min}\approx \tau^{4\beta-2}$ result together in \rf{LB}.

On the other hand, the threshold range \rf{LB} suggests that for small viscosity $\nu\ll 1$ solutions to \rf{SSM} may exhibit decrease of height at initial and intermediate times. Indeed, numerical solutions to problem \rf{SSM}, \rf{BC} indicate (see Fig. 1)  
that for initial data considered in~\cite{BT07}, that lead to the finite-time
rupture in the inviscid case, the minimum of the height tends to zero initially
until it enters the range given by \rf{LB} and only thereafter starts to converge to $M$. 
\begin{figure}
\centering
\includegraphics*[width=0.48\textwidth]{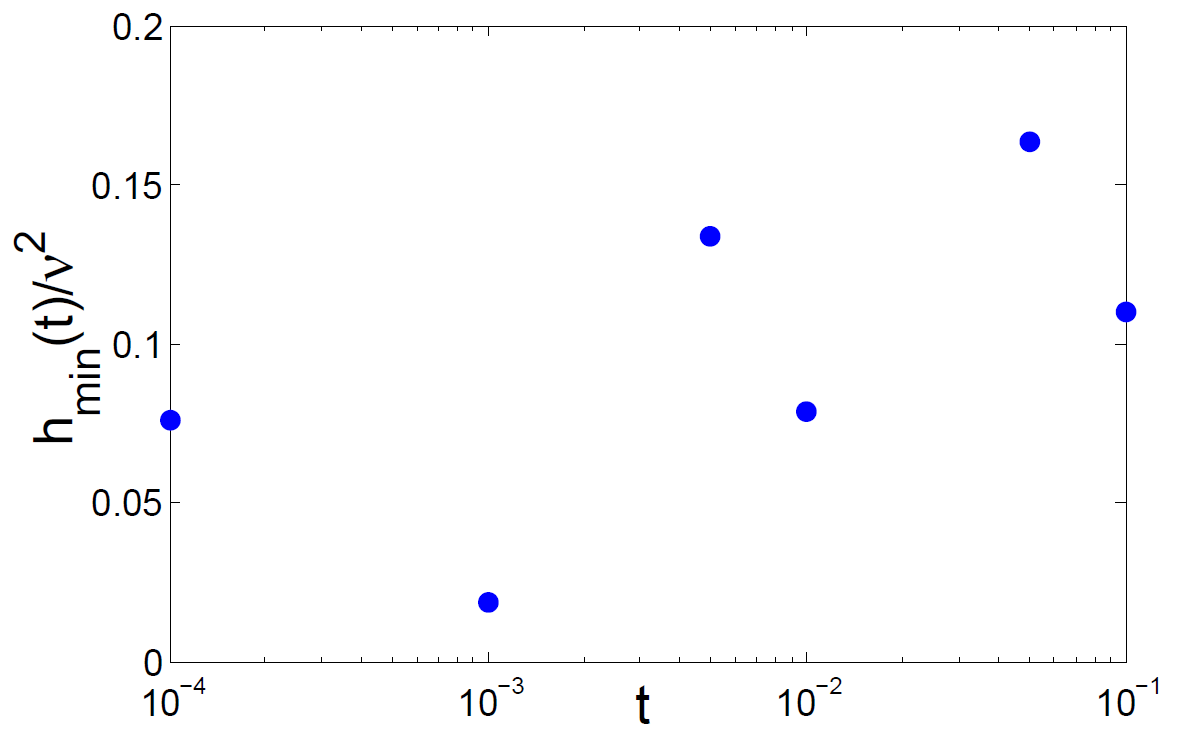}\hspace{0.5cm}
\includegraphics*[width=0.48\textwidth]{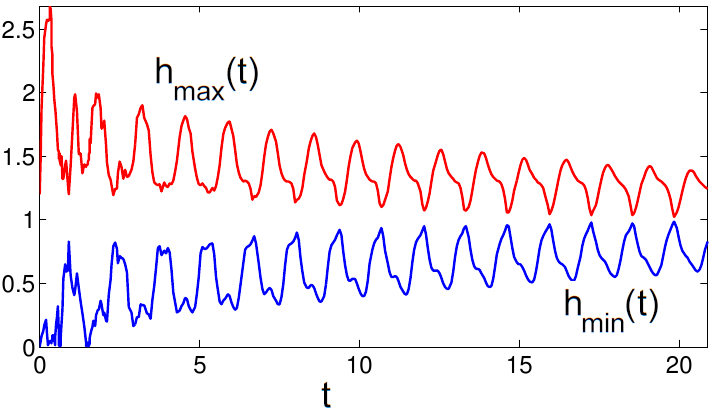}
\caption{Left: Values of minimum height divided by $\nu^2$ attained dynamically by numerical solutions to system \rf{SSM} with $\sigma=0$ and $\nu=0.1,\ 0.05,\ 0.01,\ 0.005,\ 0.001,\ 0.0001$. Initial data $h(x,0)=1-0.2\cos(\pi x/2),\ u(x,0)=\pi\sin(\pi x/2)$ was chosen as in the example of the inviscid pinch-off considered  in~\cite{BT07}.
Right: numerical evolution of the minimal sheet height in time for the same initial data and $\nu=0.1$. For obtaining numerical solutions to \rf{SSM} the numerical code developed in~\cite{KW10,KFE17} was used.}
\label{visc_threshold}
\end{figure}

Nevertheless,  up to our knowledge, mathematical results showing non-rupture
or asymptotic decay of solutions to \rf{SSM} are still lacked
in the literature. The main goal of this article is to fill this gap and to show that \rf{LB} provides the sharp uniform in space and time lower bound for the solutions to \rf{SSM} considered with $\nu>0$ and $\sigma>0$
as well as the exponential asymptotic decay of the latter to the flat profile $(h,\,v)=(M,0)$.

We start in section 2 by showing the $(H^1,\,L^2)$ exponential asymptotic decay for the non-negative solutions to \rf{SSM} to $(h,\,v)=(M,0)$. More precisely, by using the dissipation of
entropy found in~\cite{KLN11}, in the case  $\nu>0$ and $\sigma>0$ we show that any sufficiently regular solution with initially non-negative height converges to $(M,0)$ at an exponential rate (Theorem 2.1). Nevertheless, this result still does not exclude a possibility of 
rupture in finite time that could lead to loss of regularity properties of the solutions and by that terminating the asymptotic decay at the time when
the height touches the zero. Therefore, in sections 3-5 we additionally show
that the solutions to \rf{SSM} can not touch zero in finite time.
For this purpose, transformation of \rf{SSM} to the Lagrangian coordinates
following ideas suggested in~\cite{ren01,Fo03} for the systems describing viscous jets turns out to be very usefull. In
section 4, we show that system \rf{SSM} considered without curvature term
(the case $\sigma=0$) transforms to a porous medium type equation~\cite{Va07} by the Lagrangian formalism. In this case, existence, uniqueness and asymptotic  behaviour of
the solutions to it are understood completely rigorous (Theorem 4.1). In
particular, the solutions can not either blow-up or rupture in finite time. Moreover, for the special class of initial data we are
able to show exponential asymptotic decay to the flat profile in the absence
of surface tension (see estimate \rf{fct}).

In section 5 we consider the Lagrangian formulation of the full system \rf{SSM} and provide an analytical argument that general solutions to it can not rupture in finite time, but rather obey the lower bound \rf{LB}.
Finally, in Appendix we extend results of section 1 to show the exponential
asymptotic  decay of the radially symmetric solutions (Theorem A.2)
to the two-dimensional generalisation of \rf{SSM}, which appear in
applications to the axisymmetric point rupture~\cite{VLW01}.  
For this purpose, we derive for the first time, up to our knowledge, 
the entropy and energy estimates for the radially symmetric solutions (Lemma A.1).

Finally, we comment on the regularity classes of solutions considered in this article.
Results of section 2 and Appendix are rigorous and apply to suitably defined $(H^1,L^2)$ weak solutions to systems \rf{SSM} and \rf{SSMr}, respectively.
Results of section 4 are rigorous and the strong solution obtained in Theorem 4.1 is positive and unique. In section 5
we present a formal analytical argument for positive classical solutions to equation \rf{Lagr_eq} considered with general initial data.

\section{Asymptotic decay for non-negative solutions}
In this section, we show that solutions to \rf{SSM}
having non-negative height for all times asymptotically decay in $(H^1,L^2)$-norm  to the flat profile $(h,\,u)=(M,\,0)$.

First of all, we define energy and entropy as
\beas
E(v,\,h)&:=&\tfrac{1}{2}\int_0^1\left\{hv^2+\sigma h_x^2\right\}\,dx,\\
S(v,h)&:=&\tfrac{1}{2}\int_0^1 \left\{h\left(v +(G(h))_x\right)^2+ \sigma h_x^2\right\}\, dx.
\eeas
Recall the following entropy and energy equalities~\cite{KLN11} which hold for the solutions of \rf{SSM}:
\bea
\lb{EnEs_o}
\tfrac{1}{2}\tfrac{d}{dt}\int_0^1 \left\{h \left(v  + (G(h))_x \right)^2
+ \sigma h_x^2 \right\}\, dx +\nu\sigma\int_0^1 h_{xx}^2 \, dx&=&0,\\
\tfrac{1}{2}\tfrac{d}{dt}\int_0^1\left\{hv^2+\sigma h_x^2\right\}\,dx+\nu\int_0^1 hv_x^2\,dx&=&0,
\lb{EEs_o}
\eea
where $G(h):=\nu\log(h)$. 
\begin{theorem}[asymptotic exponential decay]\label{B:Th1}
Assume that initial data $(h_0,\,v_0) \in H^1(0,1) \times L^2(0,1)$, $h_0\geqslant  0$, and
\bes
\sigma>0\ \text{and}\ \nu>0.
\ees
Then there exist positive constants $A$, $B$ depending on initial data and parameters $\sigma,\,\nu$ such that the the height $h(x,t)$ asymptotically exponentially decay to zero in $H^1$-norm:
\begin{equation}\label{B:2}
 \int \limits_0^1 { h_x^2 \, dx}  \leqslant A  e^{- B\,t}
\end{equation}
for all $t \geqslant 0$, hence
\begin{equation}\label{B:uc}
h \to M \text{ in } H^1(0,\,1)\cap C^{1/2}[0,\,1] \text{ as } t \to + \infty,
\end{equation}
where $M$ is defined in \rf{CM}. Moreover, for the velocity field we have 
\be\label{B:3}
\displaystyle v\go 0\quad\text{in}\ L^2(0,\,1)\quad\text{as}\quad t\go\infty.
\ee
\end{theorem}
\begin{proof}[Proof of Theorem~\ref{B:Th1}]
Using the Poincar\'{e}  inequality
\be
\int_0^1 h_x^2 \, dx\leqslant \tfrac{1}{\pi^2}\int_0^1 h_{xx}^2  \, dx,\quad h_x(0)=h_x(1) = 0,
\lb{PEq}
\ee
from the entropy inequality \rf{EnEs_o} we find that
\begin{equation}\label{C:3}
\tfrac{\sigma}{2} \int_0^1 h_x^2 \, dx +\nu\sigma\pi^2\int_0^t\int_0^1 h_{x}^2\, dx dt  \leqslant S(v_0,h_0).
\end{equation}
By (\ref{C:3}) the  $\|h_x\|_2^2$ is dominated by the solution of
$$
y'(t) = - B\, y(t), \ \ y(0) = A\quad\text{with}\quad A =\tfrac{2}{\sigma}S(v_0,h_0),\ B = 2\nu\pi^2.
$$
Solving for $y(t)$, we deduce that
$$
||h_x||_2^2 \leqslant y(t) = A\,e^{- B\,t} \mbox{ for all } t \geqslant 0,
$$
whence (\ref{B:2}) follows. Using the Poincar\'{e}  inequality and (\ref{B:2}), we
deduce
$$
\int_0^1 (h - M)^2 \, dx \leqslant \tfrac{1}{\pi^2}
\int_0^1 {h_{x}^2  \, dx} \leqslant \tfrac{A}{\pi^2}e^{- B\,t}.
$$
From here we have
$$
|| h - M||^2_{H^1(\Omega)} \leqslant A \left( 1 + \tfrac{1}{\pi^2}\right) e^{- B\,t},
$$
whence (\ref{B:uc}) follows.
To show \rf{B:3}, first \rf{EEs_o} and the compact embedding $H^1(0,\,1)\hookrightarrow L^\infty(0,\,1)$ imply
\bes
h(x,t)\leqslant C\quad\text{for all}\ (x,\,t)\in Q_T,
\ees
where here and below $C$ denotes a generic constant depending only on
$\sigma,\,\nu$ and $h_0,\,v_0$, and one estimates
\beas
&&||(h^{3/2}v)_x||_{L^1}\leqslant ||h^{3/2}v_x||_{L^2}+2||h_x||_{L^2}||\sqrt{h}v||_{L^2}\\
&&\leqslant C\left(||\sqrt{h}v_x||_{L^2}+||h_{xx}||_{L^2}\right),
\eeas
where the last inequality follows by \rf{EnEs_o} and \rf{PEq}. 
Next, using $L^1$  Poincar\'{e}  inequality and continuous embedding $W^{1,1}(0,\,1)\hookrightarrow L^p(0,\,1)$  one gets
\be
||h^{3/2}v||_{L^p}\leqslant C_p\left(||\sqrt{h}v_x||_{L^2}+||h_{xx}||_{L^2}\right)\quad\text{for}\ p\in [1,\,\infty].
\lb{Ms4}
\ee
On the other hand, we have
\bes
C\int_0^1hv^2\,dx\geqslant  \int_0^1h^3v^2\,dx.
\ees
Using this, summing up \rf{EnEs_o} and \rf{EEs_o}, and taking into account \rf{Ms4} one arrives at
\bes
\int_0^1h^3v^2\,dx+B_1\int_0^T\int_0^1h^3v^2\,dxdt\leqslant A_1,
\ees
with $A_1$ and $B_1$ depending on $\nu,\,\sigma,\,S(v_0,h_0)$ and $E(v_0,h_0)$. 
The comparison argument then implies
\be
||h^{3/2}v||_{L^2(0,\,1)}^2\leqslant A_1\,e^{- B_1\,t}. 
\lb{u_p_decay}
\ee
The last estimate together with the fact that $h>0$ for all $t>T^*$ for some
$T^*>0$ because of \rf{B:uc} implies \rf{B:3}.
\end{proof}

\section{Transformation to Lagrangian coordinates}
Using the mass conservation \rf{CM}, we make the change of coordinates as
in~\cite[section 4.1]{PMN10}
\be
x_s(s,t)=\frac{K}{h(x(s,t),t)},\quad x(s,0)=K\int_0^s\frac{ds'}{h_0(s')},\quad x_t(s,t)=v(x(s,t),t),
\lb{x_s}
\ee
where
\bes
K=1\Big/\int_0^1\frac{1}{h_0(s)}\,ds>0.
\ees
Moreover, due to the mass conservation \rf{CM} and the second relation in
\rf{x_s} one has necessarily
\bes
K=M.
\ees
Note, that \rf{x_s} defines a monotonic mapping of $s\in[0,\,1]$ into
$x(\cdot,\,t)\in[0,\,1]$ for each $t\geqslant  0$ provided $h(x,t)>0$. 
Then equation \rf{SSM2} is trivially satisfied and equation \rf{SSM1} transforms into
\bes
x_{tt}=\frac{\sigma\,M}{x_s}\left(\frac{1}{x_s}\left(\frac{1}{x_s}\left(\frac{1}{x_s}\right)_s\right)_s\right)_s+\nu\left(\frac{x_{ts}}{x_s^2}\right)_s
\ees
Next, by  differentiating the last equation on $s$ and introducing the new function $u:=x_s$, one arrives at the equation
\be
u_{tt}=\sigma\,M\left(\frac{1}{u}\left(\frac{1}{u}\left(\frac{1}{u}\left(\frac{1}{u}\right)_s\right)_s\right)_s\right)_s-\nu\left(\frac{1}{u}\right)_{tss}.
\lb{Lagr_eq}
\ee
Note, that due to \rf{x_s} the relation between new function $u$ and the solutions of \rf{SSM} is given by
\be
u(s,t)=\frac{M}{h(x(s,t),t)}.
\lb{h_u_rel}
\ee
The boundary conditions $x(0,t)=0$ and $x(1,t)=1$, which follow from the last
relation in \rf{x_s}, imply the conservation of Lagrange function:
\be
\int_0^1u(s,t)\,ds=1.
\lb{cons_mass}
\ee
Moreover, the boundary conditions for height profile in \rf{BC},
\bes
h_x(0,t)=h_x(1,t)=0,
\ees
due to \rf{x_s} and \rf{h_u_rel}, imply the Neumann boundary conditions for $u(s,t)$:
\be
u_s(0,t)=u_s(1,t)=0.
\lb{u_BC}
\ee
Finally, using \rf{cons_mass} and \rf{u_BC} one can rewrite the
entropy and energy estimates \rf{EnEs_o}-\rf{EEs_o} in the Lagrange setting for \rf{Lagr_eq}.
Namely, the following relations hold for solutions to \rf{Lagr_eq} for all $t>0$:
\bea
\lb{EnEs}
\hspace{-0.5cm}\tfrac{1}{2}\tfrac{d}{dt}\int_0^1\left[\left(\int_0^su_t\,ds'+\nu\left(\frac{1}{u}\right)_s\right)^2+\sigma\,M^2\frac{u_s^2}{u^5}\right]\,ds&=&-\nu\sigma\,M^2\int_0^1\left(\frac{3u_s^2}{u^5}-\frac{u_{ss}}{u^4}\right)^2u\,ds,\\
\tfrac{1}{2}\tfrac{d}{dt}\int_0^1\left[\left(\int_0^su_t\,ds'\right)^2+\sigma\,M^2\frac{u_s^2}{u^5}\right]\,ds&=&-\nu\int_0^1\left|\frac{u_t}{u}\right|^2\,ds.
\lb{EEs}
\eea
Although, a priori estimates \rf{EnEs}--\rf{EEs} can be derived for problem \rf{Lagr_eq}, \rf{u_BC} independently, the equivalence of Euler and Lagrange formulations given by transformation \rf{x_s} 
for classical solutions with positive for all times height  implies them directly. We note also that this equivalence valid also for local solutions with positive $h(x,t)$ for all $t\in (0,\,t_r)$, where $t_r>0$ is the first time at which $h(x,t)$ touches zero. In particular, this observation is
important for application of \rf{EnEs}--\rf{EEs} in the vicinity of potential rupture in section 5.

\section{Boundedness and asymptotic decay in the case $\sigma=0$}

In absence of surface tension, equation \rf{Lagr_eq} after integration in time reduces to 
\be
u_t+\nu\left(\frac{1}{u}\right)_{ss}=f(s)\quad\text{or}\quad u_t=\nu\left(u^{-2}u_s\right)_s+f(s),
\lb{Lagr_eq_zcurv}
\ee
where the integration factor is determined by the initial conditions and \rf{x_s} as
\be
f(s)=\frac{M}{h_0(x(s,0))}\left[v_0(x(s,0))+\nu\frac{h_{0,x}(x(s,0))}{h_0(x(s,0))}\right]_x.
\lb{f_def}
\ee
For compatibility with boundary conditions \rf{u_BC} and conservation of
Lagrange function \rf{cons_mass} one has to impose the following integral condition onto initial data:
\be
\int_0^1f(s)\,ds=0,
\lb{zero_f}
\ee
which is immediately true due to \rf{f_def} and \rf{x_s} if the initial data
is compatible with \rf{BC}, i.e.
\bes
h_{0,x}=v_0=0\quad\text{at}\quad x=0,\,1.
\ees
We note that equation \rf{Lagr_eq_zcurv} coincides with the porous medium equation 
in~\cite[(1.27)]{Va06} with exponent $m=-1$, which arises from the backward
parabolic {\it super fast diffusion equation}
\bes
u_t=\Delta u^m
\ees
after reversing the time in the latter. Therefore, for investigation of regularity and
qualitative behaviour of solutions to \rf{Lagr_eq_zcurv} in the next theorem
we apply the methods developed for porous medium equations in~\cite{Va07,Va06}.
\begin{theorem}
Suppose,
\be
u_0,\,f\in L^\infty(0,1)\quad\text{and}\quad u_0>0,
\lb{ID}
\ee
then a unique strong solution $u(s,t)$ to problem
\rf{Lagr_eq_zcurv}, \rf{u_BC} with initial data 
\bes
u(s,0)=u_0(s)=M/h_0(x(s,0))
\ees
exists for all $t>0$. For it one has a linear in time {\bf upper} and uniform {\bf lower} bounds:
\be
0<C\leqslant u(s,t)\leqslant||u_0||_\infty+||f||_\infty t\quad\text{for}\quad(x,t)\in(0,\,1)\times(0,\infty),
\lb{MaxPr}
\ee
where constant $C$ depends only on $u_0$. Hence, $u(s,t)$ can not blow up in finite time or touch zero. Moreover, if $f=0$ then there
exist positive constants
\be
A:=\int_0^1  (u_0 -1)^2\,ds,\quad  B:=2\nu\left(\frac{\ln||u_0||_\infty} {||u_0||_\infty - 1}\right)^2
\lb{AB_def}
\ee
such that
\be
||u(\cdot,t)-1||_{L^2(0,\,1)}^2\leqslant A\exp\{-Bt\}\quad\text{for all}\ t>0
\lb{fct}
\ee 
holds, i.e. $L^2$-exponetial asymptotic decay of $u$ to the flat profile.
\lb{Th4.1}
\end{theorem}
\begin{proof}[Proof of Theorem~\ref{Th4.1}] Existence part follows by the quasilinear parabolic theory~\cite{LSU68} and the fact that the main term in \rf{Lagr_eq_zcurv}, namely
\bes
\frac{u_{ss}}{u^2} 
\ees
is uniformly coercive unless $u$ becomes infinite at some point. In
particular, the upper bound \rf{MaxPr} follows by an application of the
maximal principle~\cite[Lemma 3.3, p. 34]{Va07}. Note, that this upper bound does not depend on viscosity $\nu$. Indeed, this can be
checked by rescaling time $\bar{t}=\nu t,\ \bar{f}=\nu f$ and applying the maximal principle.
For the global well-posedness of the solution for all $t>0$ it remains to show the uniform lower bound on $u(s,t)$.

First, we note that the entropy and energy equalities reduce for \rf{Lagr_eq_zcurv} to
\bea
\int_0^1\left(\int_0^su_t\,ds'+\nu\left(\frac{1}{u}\right)_s\right)^2\,ds&=&\const,\\
\tfrac{1}{2}\tfrac{d}{dt}\int_0^1\left(\int_0^su_t\,ds'\right)^2\,ds&=&-\nu\int_0^1\left|\frac{u_t}{u}\right|^2\,ds.
\eea
Combining them together and using the inequality $(y+z)^2\geqslant  y^2/2-z^2$ one estimates
\bes
\int_0^1\left(\int_0^su_t\,ds'+\nu\left(\frac{1}{u}\right)_s\right)^2\,ds\geqslant \tfrac{\nu^2}{2}\int_0^1\left|\left(\frac{1}{u}\right)_s\right|^2\,ds-\int_0^1\left(\int_0^su_t\,ds'\right)^2\,ds,
\ees
and therefore, one concludes that
\be
\displaystyle \sup_{t\in(0,\,\infty)}\int_0^1\left|\left(\frac{1}{u}\right)_s\right|^2\,ds\leqslant\const.
\lb{one_u_s}
\ee
Next inequality can be considered as an energy type one for \rf{Lagr_eq_zcurv}
(see~\cite{Va07,Va06}). Testing  \rf{Lagr_eq_zcurv} with $-1/u^2$ 
and subsequently integrating by parts provides the equality
\be
\tfrac{d}{dt}\int_0^1\frac{1}{u}\,ds+2\nu\int_0^1\frac{u_s^2}{u^5}\,ds=-\int_0^1f(s)\left(\frac{1}{u^2}-\int_0^1\frac{ds'}{u^2}\right)\,dsdt,
\lb{EEs_PME}
\ee
where we have used also \rf{zero_f}.
Let us integrate the last equality in time and estimate the right-hand side of it using $L^1$ Poincare inequality as follows.
\beas
&&\Big|\int_0^T\int_0^1f(s)\left(\frac{1}{u^2}-\int_0^1\frac{ds'}{u^2}\right)\,dsdt\Big|\leqslant||f||_\infty\int_0^T\int_0^1\left|\left(\frac{1}{u}\right)_s\right|\,dsdt=||f||_\infty\int_0^T\int_0^1\left|\frac{u_s}{u^3}\right|\,dsdt\\
&&\displaystyle\leqslant ||f||_\infty\int_0^T\sqrt{\int_0^1\frac{u_s^2}{u^5}\,ds}\sqrt{\int_0^1\frac{1}{u}\,ds}dt\leqslant\frac{||f||_\infty}{2}\left[\eps\sup_{t\in(0,\,T]}\int_0^1\frac{1}{u}\,ds+\frac{T}{\eps}\int_0^T\int_0^1\frac{u_s^2}{u^5}\,dsdt\right].
\eeas
Combining this with \rf{EEs_PME} implies for $T=\eps^2=1/||f||_\infty$ that
\bes
\displaystyle\sup_{t\in(0,\,T]}\int_0^1\frac{1}{u}\,ds+C(\nu,||f||_\infty)\iint_{Q_T}\left|\left(\frac{1}{u}\right)_s\right|^2\,dsdt\leqslant\int_0^1\frac{1}{u_0}\,ds.
\ees
But the latter inequality can be extended to
\be
\displaystyle\sup_{t\in(0,\,\infty)}\int_0^1\frac{1}{u}\,ds\leqslant\int_0^1\frac{1}{u_0}\,ds
\lb{one_u}
\ee
via repeating the last argument iteratively on small time intervals $(t,t+1/||f||_\infty^2)$.
We conclude that estimates \rf{one_u_s} and \rf{one_u} imply together that
\bes
\Big|\Big|\frac{1}{u}\Big|\Big|_{L^\infty(0,\infty;W^{1,1}(0,1))}\leqslant\const.
\ees
In turn, the compact embedding $W^{1,1}(0,\,1)\hookrightarrow L^\infty(0,\,1)$ implies
\bes
\Big|\Big|\frac{1}{u}\Big|\Big|_{L^\infty((0,\,1)\times(0,\infty))}\leqslant C,
\ees
i.e. the uniform lower bound in \rf{MaxPr}.

It remains to show \rf{fct} when $f=0$. For that we test \rf{Lagr_eq_zcurv}
with $u$ and integrate by parts to get:
\be
\tfrac{1}{2}\tfrac{d}{dt}\int_0^1 u^2\,ds+\nu\int_0^1 (\ln u)_s^2\,ds=0.
\lb{EE_a_1}
\ee
Using Poincare inequality and \rf{cons_mass} one gets 
\bes
\int_0^1 \left|\ln (u) )\right|^2 ds \leqslant  \int_0^1 (\ln(u))_s^2 ds.
\ees
Next, using \rf{MaxPr} one estimates
\bes
\left|\ln (u(x,t)) \right|^2 \geqslant  \left|\frac{\ln||u_0||_\infty} {||u_0||_\infty-1}\right|^2 (u - 1)^2\quad\text{for all}\quad (x,\,t)\in Q_T.
\ees
Combining the last two estimates one gets from \rf{EE_a_1}:
\bes
\int_0^1 ( u -1 )^2 ds  + B \int_0^t \int_0^1 ( u-1)^2\, dsdt  \leqslant   A
\ees
with $A$ and $B$ defined as in \rf{AB_def}. Here we have used also the equality
\bes
\int_0^1  u^2 ds = \int_0^1 (u -1)^2 ds + 1.
\ees
Whence by comparing $y(t):=\int_0^1 ( u -1 )^2 ds$ to the solution $\bar{y}(t)$ of the ODE
\bes
\bar{y}'(t)+B\bar{y}(t)=0,\quad\bar{y}(0)=A,
\ees
one obtains 
\bes
0\leqslant y(t)\leqslant A\exp\{-Bt\}\quad\text{for all}\ t>0.
\ees
The last estimate gives \rf{fct}.
\end{proof}
\begin{remark}
The estimates \rf{MaxPr} together with \rf{h_u_rel} show for
the original system \rf{SSM} considered with general initial data and $\sigma=0$ that the height profile stays bounded for all
times from above and below, respectively, but do not preclude it from approaching zero in infinite time with the rate independent of $\nu$.
Note that the bound \rf{MaxPr} holds also in the limiting case of
\rf{Lagr_eq_zcurv}  considered with $\nu=0$, which can be shown in this case via an explicit integration. 
\end{remark}
\begin{remark}
We note that estimate \rf{fct} implies the exponential asymptotic decay to
constant profile $(h,\,v)=(M,\,0)$ for a special class of solutions to original viscous sheet system \rf{SSM}.
Indeed, using Lagrangian transformation \rf{x_s} and formula \rf{f_def} one obtains that any solution to \rf{SSM} having initial data satisfying
\be
v_0(x)=-\nu\frac{h_{0x}(x)}{h_0(x)}\quad\text{and}\quad h_0>0
\lb{CompCond}
\ee
should decay to  $(h,\,v)=(M,\,0)$ in finite time. 
Relation \rf{CompCond} written in Lagrangian coordinates takes the form:
\be
v_0(x(s,0))=-\tfrac{\nu}{M} h_{0s}(x(s,0)).
\lb{CompCondl}
\ee
Figures 2-3 present numerical confirmation of the convergence to the flat profile for $\nu=1$. In the first row we have taken special initial data
\bea
h_0(x(s,\,0))&=&\cos(\pi s)+\pi,\quad v_0(x(s,\,0))=\frac{\pi}{\sqrt{\pi^2-1}}\sin(\pi s),\nonumber\\
x(s,\,0)&=&M\int_0^s\frac{ds'}{h_0(x(s',\,0))}=\frac{2}{\pi}\arctan\left(\frac{\tan(\frac{\pi}{2}s)(\pi-1)}{\sqrt{\pi^2-1}}\right)
\lb{flat}
\eea
which satisfy \rf{CompCond} and \rf{CompCondl}. The solution to \rf{SSM} considered with $\sigma=0$ and initial data defined by \rf{flat} converges to the flat profile $(h,v)=(\sqrt{\pi^2-1},\,0)$. 

Next, we slightly modified the initial data \rf{flat} to
\bea
h_0(x(s,\,0))&=&\cos(\pi s)+\pi,\quad v_0(x(s,\,0))=\pi\sin(\pi s),\nonumber\\
x(s,\,0)&=&M\int_0^s\frac{ds'}{h_0(x(s',\,0))}=\frac{2}{\pi}\arctan\left(\frac{\tan(\frac{\pi}{2}s)(\pi-1)}{\sqrt{\pi^2-1}}\right),
\lb{non_flat}
\eea
so that in this case
\be
f(s)=\pi^2\cos(\pi s)\left(1-\frac{1}{\sqrt{\pi^2-1}}\right)\not =0.
\lb{f_ex}
\ee
The second row in Fig. 2 shows the dynamical snapshots of the solution to \rf{SSM} with initial data given by \rf{non_flat}.
The solution converges then in finite time to the stationary solution of \rf{Lagr_eq_zcurv} satisfying the ODE:
\be
\nu\left(\frac{1}{u_\infty}\right)_{ss}=f(s).
\lb{SS}
\ee
The last ODE with the right-hand side \rf{f_ex}, $\nu=1$ and boundary conditions \rf{u_BC} can be integrated to give explicitly
$u_\infty$ and the corresponding limiting height profile as
\bea
h_\infty(x_\infty(s))&=&\frac{M}{u_\infty(s)}=-M\left[\cos(\pi s)\left(1-\frac{1}{\sqrt{\pi^2-1}}\right)+C_0\right],\nonumber\\
x_\infty(s)&=&M\int_0^s\frac{ds'}{h_\infty(x_\infty(s'))}=-\int_0^s\frac{ds'}{\cos(\pi s')\left(1-\frac{1}{\sqrt{\pi^2-1}}\right)+C_0},
\lb{final_non_flat} 
\eea
where $C_0\approx -1.2$ is fixed by the conservation of mass laws \rf{CM} or equivalently \rf{cons_mass}.

Finally, in Fig. 3 we present the numerically calculated decay of $||u(s,t)-1||_2$, with
\bes
u(s,t)=\frac{M}{h(x(s,t),t)}
\ees
according to \rf{h_u_rel}, for the solution to \rf{SSM} with initial data  \rf{flat}. The decay is profoundly exponential with the numerical saturation effect
when $||u(s,t)-1||_2$ becomes small. Comparison with the derived analytical bound \rf{fct} shows that the latter is the upper bound but is naturally not sharp.
\begin{figure}
\centering
\includegraphics*[width=0.38\textwidth]{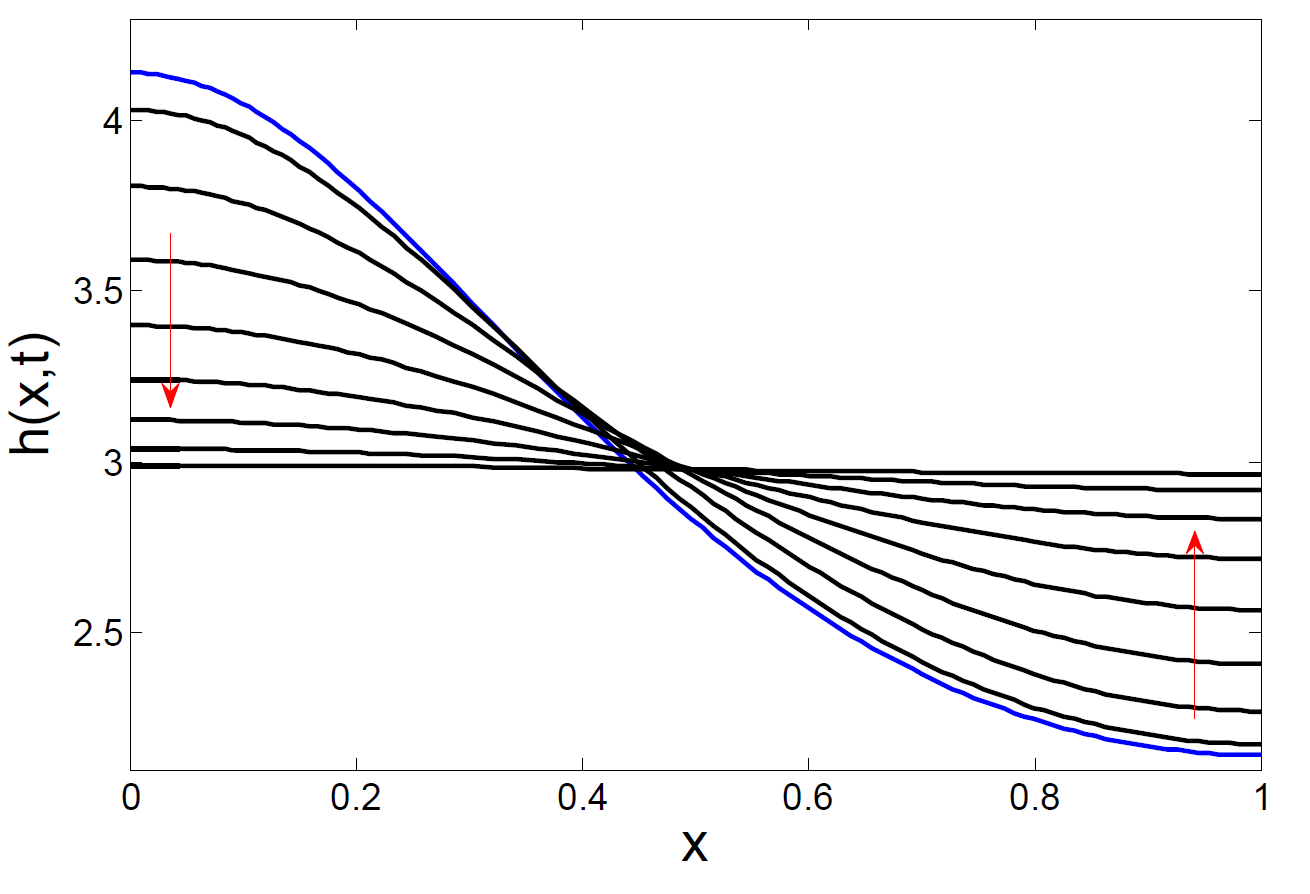}
\includegraphics*[width=0.37\textwidth]{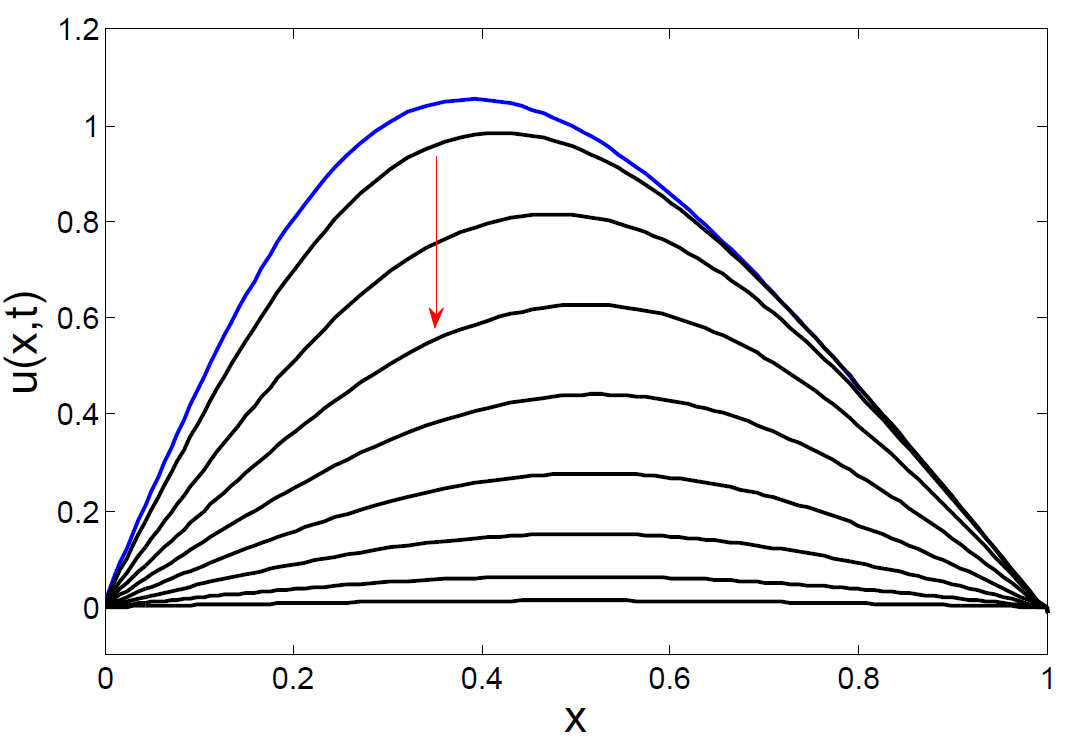}
\includegraphics*[width=0.37\textwidth]{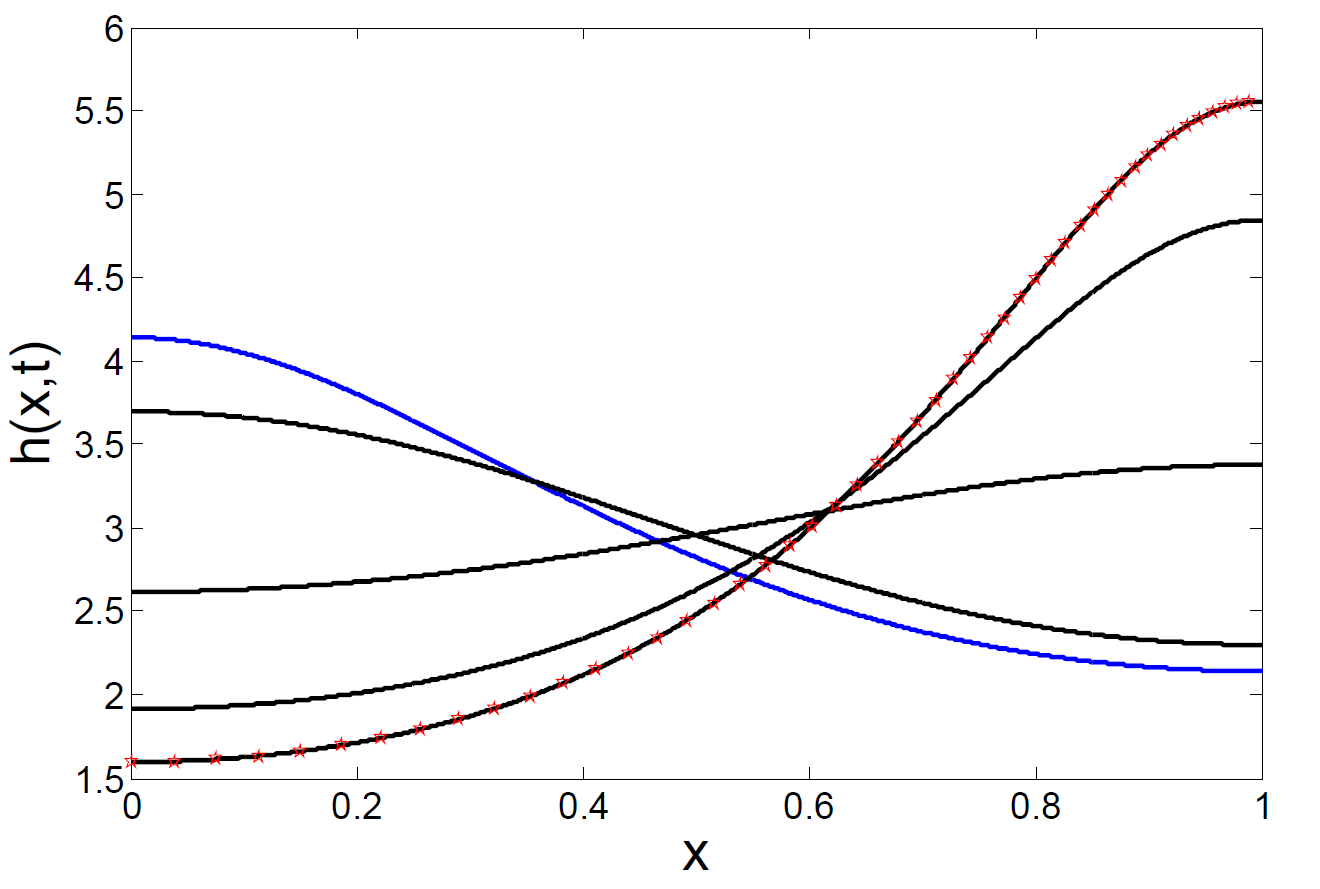}
\includegraphics*[width=0.37\textwidth]{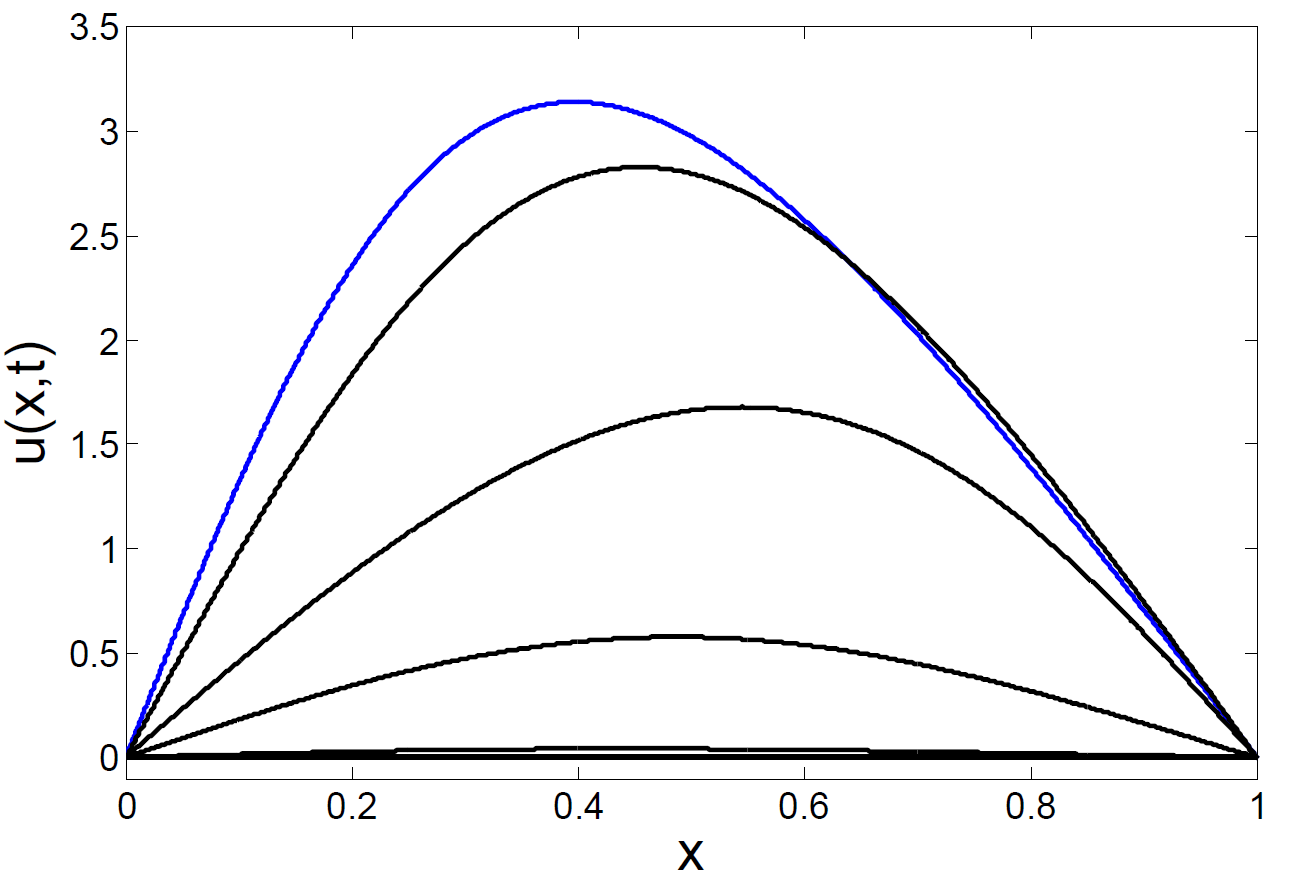}
\caption{Convergence to stationary profiles of solutions to \rf{SSM} with
  $\nu=1$ and $\sigma=0$ illustrated by several snapshots of the height (left column) and velocity profiles (right column).
First row: convergence to the flat profile $(M,\,0)$ for  the initial data \rf{flat}. Second row: convergence to nonstationary height profile $h_\infty$
with $v_\infty=0$ for the initial data \rf{non_flat}. The initial data in both
cases is depicted by blue solid lines. The red dotted curve corresponds to the
analytical formula \rf{final_non_flat} for $h_\infty(x)$ and is superimposed
onto the final numerical height profile in the second simulation.
For obtaining numerical solutions to \rf{SSM} the numerical code developed in~\cite{KW10,KFE17} was used.}
\label{pinch_ex}
\end{figure}
\begin{figure}
\centering
\includegraphics*[width=0.5\textwidth]{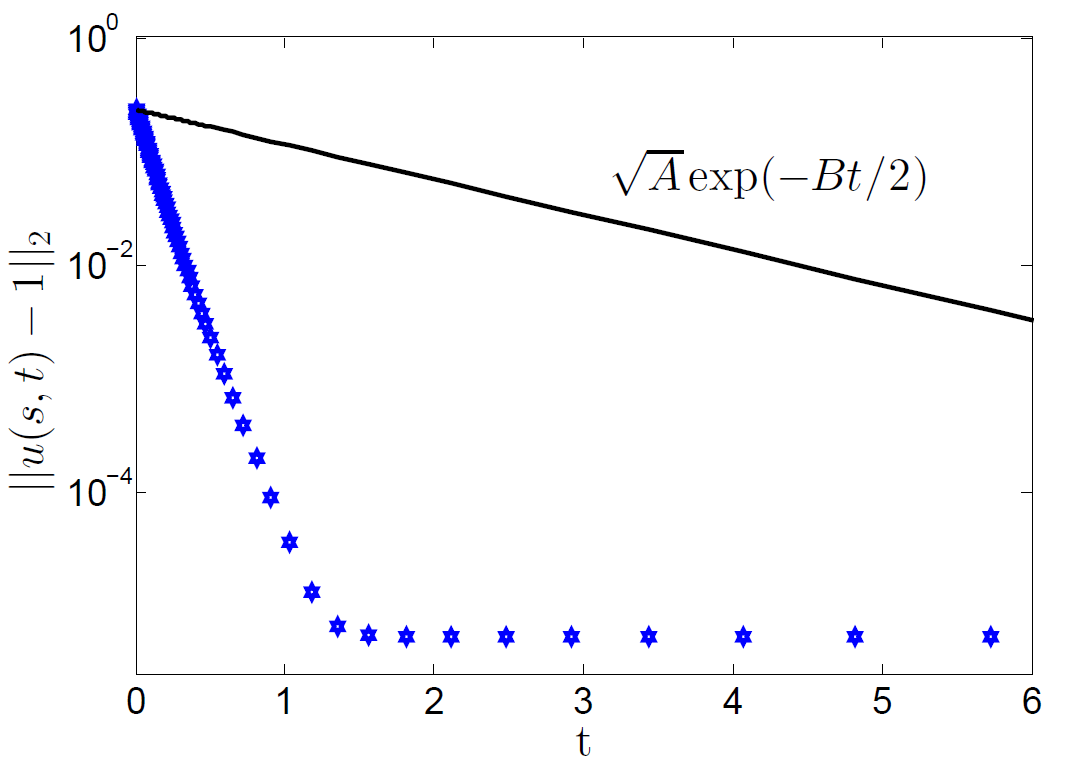}
\caption{Semilog plot (blue markers) of the decay of $||u(s,t)-1||_2$ corresponding to the solution to \rf{SSM} with initial data  \rf{flat}. 
Solid line indicates the upper bound \rf{fct} with $||u_0||_\infty=\sqrt{\pi^2-1}/(\pi-1)$ and $B=1.4252,\ A=0.0549$. For obtaining numerical solutions to \rf{SSM} the numerical code developed in~\cite{KW10,KFE17} was used.}
\label{pinch_ex}
\end{figure}

\end{remark}

\section{Non-rupture in the case $\sigma>0$.}
We consider now the full equation \rf{Lagr_eq}. We assume without loss of generality $\sigma\,M=1$, by rescaling the time variable, and rewrite it in terms of the variable $h(s,t):=h(x(s,t),t)$ using \rf{h_u_rel} as
\bes
-\left(\frac{h_{t}}{h^2}\right)_t+\nu h_{tss}=\left(h\left(h\left(hh_s\right)_s\right)_s\right)_s.
\ees
The last equation is equivalent to
\be
-\left(\frac{h_{t}}{h^2}\right)_t+\nu h_{tss}=\left(\frac{h^2}{2}h_s^2+h^3h_{ss}\right)_{ss}.
\lb{Lagr_eq_h}
\ee
Note that from the energy and entropy inequalities \rf{EnEs}-\rf{EEs}, using again \rf{h_u_rel}, we have a priori bounds:
\be
||h_s||_{L^{\infty}(0,T;L^2(0,1))}\leqslant \frac{C}{\nu}\quad\text{and}\quad||h_t||_{L^2(Q_T)}\leqslant  \frac{C}{\sqrt{\nu}}.
\lb{HEst_1}
\ee
For showing the second estimate in \rf{HEst_1} we have used additionally equivalence of the Euler \rf{SSM} and Lagrange \rf{Lagr_eq_h} systems provided by transformation \rf{x_s}
and the uniform upper bound
\bes
h(s,t)\leqslant C
\ees
following from \rf{EEs_o}. In \rf{HEst_1} and below in this section $C$ denotes a constant that may depend on the initial data only, but not on the parameters $\nu$ and $T$ of the problem.
Using the first estimate in \rf{HEst_1} one naturally gets the uniform H\"older continuity in space:
\be
|h(s_1,t)-h(s_2,t)|\leqslant  \frac{C}{\nu}|s_1-s_2|^{1/2}.
\lb{HEst_2}
\ee
Next, applying~\cite[ Lemma 7.19]{Va07}, in particular (7.32) there, one obtains from \rf{HEst_1}
\be
|h(s,t_1)-h(s,t_2)|\leqslant C||h_s||_{L^{\infty}(0,T;L^2(0,1))}||h_t||_{L^2(Q_T)}|t_1-t_2|^{1/4}\leqslant  \frac{C}{\nu^{3/2}}|t_1-t_2|^{1/4}.
\lb{HEst_3}
\ee
From the last two estimates we conclude $h\in C_{s,t}^{1/2,\,1/4}(\bar{Q}_T)$. Note that a priori estimates \rf{HEst_1}--\rf{HEst_3} do not depend on smallness of $h$
and, therefore, should be also valid in the vicinity of a rupture point if it happens.

Below we provide a formal analytical  argument for the fact that solutions to \rf{Lagr_eq_h}  considered with general initial data can not rupture in finite time.
The proof proceeds by contradiction. We will assume that there is a solution to  \rf{Lagr_eq_h} that ruptures in finite time.
Let us integrate  \rf{Lagr_eq_h} once in time and write it as
\be
h_t=\nu h^2h_{ss}-h^2\int_0^t\left[\frac{h^2}{2}h_s^2+h^3h_{ss}\right]_{ss}\,d\tau-h^2f(s),
\lb{int_eq}
\ee
where
\bes
f(s):=-h_t(s,0)/h^2(s,0)+\nu h_{ss}(s,0).
\ees
Close to the expected rupture time $t=t_r$ we will zoom into the neighborhood
of the pinch-off point and assume that there the rupture solution to
\rf{Lagr_eq_h} is well approximated by the solution to
\be
h_t=\nu\eps^2h_{ss}-\eps^2\left(\int_0^t\left[\frac{h^2}{2}h_s^2+h^3h_{ss}\right]_{ss}\,d\tau-f(s)\right),
\lb{appr_eq}
\ee
considered now in the whole real line $s\in\mR$ via a suitable extension of $h$ to a small constant profile outside of the pinch-off region.
Additionally, we will assume that the following analogs of a priori bounds \rf{HEst_1}--\rf{HEst_3} 
\bea
\lb{HEst_1a}
||h_s||_{L^{\infty}(0,T;L^2(\mR))}&\leqslant&\frac{C}{\nu},\\
|h(s_1,t)-h(s_2,t)|&\leqslant& \frac{C}{\nu}|s_1-s_2|^{1/2},\ |h(s,t_1)-h(s,t_2)|\leqslant \frac{C}{\nu^{3/2}}|t_1-t_2|^{1/4},\nonumber\\
&&\text{for all}\ s_1,\,s_2\in\mR\ \text{and}\ t_1,\,t_2\in[0,\,t_r(\eps)),
\lb{HEst_2a}
\eea
still hold for the approximating solution to \rf{appr_eq}, where constant $C$ does not depend on $\eps$.
Next, by applying an appropriate time shift we assume that the initial data for \rf{appr_eq} satisfies the uniform spatial bounds:
\be
0<\eps\leqslant h(s,0)=h_0(s)\leqslant C\eps\quad\text{with}\quad C>1.
\lb{Indata}
\ee
Note, that the above time shift implies also that the finite rupture time of the solution to \rf{appr_eq}
\be
t_r=t_r(\eps)\go 0\quad\text{as}\quad \eps\go 0.
\lb{t_r}
\ee
Naturally, for the small initial data \rf{Indata}  and times close to
$t=t_r(\eps)$ one would expect the solution of the approximate equation
\rf{appr_eq} stay close to the rupture solution to the original equation
\rf{int_eq}. Finally, we will make a technical assumption that the approximate
solution to \rf{appr_eq} satisfying \rf{Indata} stays of the same order
smallness up to the rupture time, i.e.
\be
h(s,t)\leqslant C\eps\quad\text{for}\quad t\in[0,\,t_r(\eps)).
\lb{small_bnd}
\ee
Let us now write the exact integral representation of solutions to \rf{appr_eq} :
\bea
h(s,t)&=&\int_{-\infty}^{\infty}K_{\nu,\eps}(s-\xi,t)h_0(\xi)\,d\xi\nonumber\\
&-&\eps^2\int_0^t\int_{-\infty}^{\infty}K_{\nu,\eps}(s-\xi,t-\tau)\left(\int_0^\tau\left[\frac{h^2}{2}h_\xi^2+h^3h_{\xi\xi}\right]_{\xi\xi}\,d\tilde{\tau}+f(s)\right)\,d\xi d\tau,
\lb{HeatRepr}
\eea
where $K_{\nu,\eps}(s,t)$ denotes the heat kernel:
\bes
K_{\nu,\eps}(s,t):=\frac{1}{\sqrt{4\pi\nu\eps^2t}}\exp\left[-\frac{s^2}{4\nu\eps^2t}\right]
\ees
having the well-known property
\bes
\frac{\partial K_{\nu,\eps}}{\partial t}=\nu\eps^2\frac{\partial^2 K_{\nu,\eps}}{\partial\xi^2}.
\ees
Using assumption \rf{Indata} one can show that the following estimate from below for the first integral in \rf{HeatRepr}:
\be
\int_{-\infty}^{\infty}K_{\nu,\eps}(s-\xi,t)h_0(\xi)\,d\xi\geqslant  \frac{\eps}{\sqrt{\pi}}\int_{-\infty}^{\infty}\exp[-r^2]\,dr=\eps.
\lb{Indata_est}
\ee
In the rest of the proof we will estimate the second integral in \rf{HeatRepr}
from above and show that contributions from the nonlinear curvature term
become smaller then the right hand-side of \rf{Indata_est} when 
$\eps$ is sufficiently small and satifies the relation \rf{eps_nu} below. First, one estimates
\bea
\eps^2\Big|\int_0^t\int_{-\infty}^{\infty}K_{\nu,\eps}(s-\xi,t-\tau)f(s)\,d\xi d\tau\Big|\leqslant\eps^2||f(s)||_\infty\frac{t}{\sqrt{\pi}}\int_{-\infty}^{\infty}\exp[-r^2]\,dr=\eps^2t||f(s)||_\infty.
\lb{Indata_est1}
\eea
Next, one shows using several times integration by parts that
\bea
&&\eps^2\int_0^t\int_{-\infty}^{\infty}K_{\nu,\eps}(s-\xi,t-\tau)\int_0^\tau\left[\frac{h^2}{2}h_\xi^2+h^3h_{\xi\xi}\right]_{\xi\xi}\,d\tilde{\tau}d\xi d\tau=\nonumber\\
&&-\frac{1}{\nu}\int_0^t\int_{-\infty}^{\infty}\frac{\partial K_{\nu,\eps}(s-\xi,t-\tau)}{\partial\tau}\int_0^\tau\left[\frac{h^2}{2}h_\xi^2+h^3h_{\xi\xi}\right]\,d\tilde{\tau}d\xi d\tau\nonumber\\
&&=\frac{1}{\nu}\int_0^t\int_{-\infty}^{\infty}K_{\nu,\eps}(s-\xi,t-\tau)\left[\frac{h^2}{2}h_\xi^2+h^3h_{\xi\xi}\right]\,d\xi d\tau\nonumber\\
&&=-\frac{5}{2\nu}\int_0^t\int_{-\infty}^{\infty}K_{\nu,\eps}(s-\xi,t-\tau)h^2h_\xi^2\,d\xi d\tau-\frac{1}{4\nu^2\eps^2}\int_0^t\int_{-\infty}^{\infty}\frac{\partial K_{\nu,\eps}(s-\xi,t-\tau)}{\partial\tau}h^4\,d\xi d\tau\nonumber\\
&&=-\frac{5}{2\nu}\int_0^t\int_{-\infty}^{\infty}K_{\nu,\eps}(s-\xi,t-\tau)h^2h_\xi^2\,d\xi d\tau-\frac{1}{8\nu^2\eps^2}\int_0^t\int_{-\infty}^{\infty}K_{\nu,\eps}(s-\xi,t-\tau)\frac{h^4}{(t-\tau)}\,d\xi d\tau\nonumber\\
&&+\frac{1}{4\nu^2\eps^2}\int_0^t\int_{-\infty}^{\infty}K_{\nu,\eps}(s-\xi,t-\tau)\frac{h^4}{(t-\tau)}\frac{(s-\xi)^2}{4\nu\eps^2(t-\tau)}\,d\xi d\tau\nonumber\\
&&=:I_1+I_2+I_3.
\lb{RHS_est}
\eea
Let us estimate consequently the integrals $I_i,\,i=1,2,3$ in \rf{RHS_est} at the special point $s=s_r$, where the rupture occurs at time $t=t_r$.
At this point, using the H\"older estimates \rf{HEst_2a} one obtains
\be
h(\xi,\tau)\leqslant h(s_r,\tau)+\frac{C}{\nu}|\xi-s_r|^{1/2}\leqslant \frac{C}{\nu^{3/2}}|t_r-\tau|^{1/4}+\frac{C}{\nu}|\xi-s_r|^{1/2}.
\lb{r_p_est}
\ee
First, using \rf{HEst_1a} and \rf{small_bnd} one estimates at $t=t_r(\eps)$
\be
|I_1|=\frac{5}{2\nu}\int_0^t\int_{-\infty}^{\infty}K_{\nu,\eps}(s_r-\xi,t-\tau)h^2h_\xi^2\,d\xi d\tau\leqslant \frac{C\eps}{\sqrt{\nu^3}}\int_0^t\frac{d\tau}{\sqrt{t-\tau}}\int_{-\infty}^{\infty}h_\xi^2\,d\xi\leqslant\frac{C\eps\sqrt{t}}{\sqrt{\nu^3}}.
\lb{RHS_est_1}
\ee
Next, using \rf{r_p_est} and \rf{small_bnd} one obtains
\bea
|I_2|&=&\frac{1}{8\nu^2\eps^2}\int_0^t\int_{-\infty}^{\infty}K_{\nu,\eps}(s_r-\xi,t-\tau)\frac{h^4(\xi,\,\tau)}{(t-\tau)}\,d\xi d\tau\nonumber\\
&\leqslant&\frac{C\eps}{\nu^2}\int_0^t\int_{-\infty}^{\infty}K_{\nu,\eps}(s_r-\xi,t-\tau)\frac{h(\xi,\,\tau)}{(t-\tau)}\,d\xi d\tau\nonumber\\
&\leqslant&\frac{C\eps}{\nu^2}\int_0^t\frac{1}{(t-\tau)^{3/4}}\frac{1}{\nu^{3/2}}\int_{-\infty}^{\infty}K_{\nu,\eps}(s_r-\xi,t-\tau)\,d\xi d\tau\nonumber\\
&+&\frac{C\eps}{\nu^2}\int_0^t\frac{1}{(t-\tau)^{3/4}}\frac{(4\nu\eps^2)^{1/4}}{\nu}\int_{-\infty}^{\infty}K_{\nu,\eps}(s_r-\xi,t-\tau)\sqrt{\frac{|\xi-s_r|}{\sqrt{(t-\tau)4\nu\eps^2}}}\,d\xi d\tau\nonumber\\
&\leqslant&\frac{C\eps}{\nu^2}\int_0^t\frac{1}{(t-\tau)^{3/4}}\left( \frac{1}{\nu^{3/2}}+\frac{(4\nu\eps^2)^{1/4}}{\nu}\right)d\tau\leqslant\frac{C\eps t^{1/4}}{\nu^{7/2}}\quad\text{for}\ \eps<1\ \text{and}\ \nu<1.
\lb{RHS_est_2}
\eea
Finally, proceeding analogously as in \rf{RHS_est_2} one shows that
\be
|I_3|\leqslant\frac{C\eps t^{1/4}}{\nu^{7/2}}
\lb{RHS_est_3}
\ee
Therefore, combining \rf{RHS_est_1}-\rf{RHS_est_3} with \rf{t_r} and taking $\eps$ so small that 
\be
t_r(\eps)<C_1\nu^{14}
\lb{t_r_nu}
\ee
one obtains that
\bes
|I_1+I_2+I_3|\leqslant\frac{\eps}{2}.
\ees
The last estimate combined with \rf{Indata_est1}  and  \rf{Indata_est}  implies by \rf{HeatRepr} that for all sufficiently small $\eps>0$
\bes
h(s_r,t)\geqslant  \frac{\eps}{4}\quad\text{for all}\quad t\in[0,\,t_r(\eps)],
\ees
which contradicts to the initial assumption of the rupture of the solution to \rf{appr_eq} at the finite time $t=t_r(\eps)$.
\begin{remark}
We note, a posteriori, that the estimate on the rupture time \rf{t_r_nu} together with the second estimate in \rf{HEst_2a} and assumption on the initial data \rf{Indata}
imply that
\be
0<\eps\leqslant h(s_r,0)\leqslant \frac{Ct_r(\eps)^{1/4}}{\nu^{3/2}}\leqslant C\nu^2.
\lb{eps_nu}
\ee
This condition on the minimal height at which the viscous term starts to dominate the curvature term is consistent both with the self-similar scalings~\cite{EF15} and numerical observations.
\end{remark}

\section{Discussion}

This article presents several interesting analytical results on qualitative behaviour of solutions to viscous sheet system \rf{SSM} and its Lagrangian counterpart \rf{Lagr_eq}. 
It poses also several interesting open questions summarised below: 
\begin{itemize}
\item Theorem 2.1 allows for $(H^1,\,L^2)$ weak solutions with non-negative initial data. It would be interesting to investigate implications of estimates \rf{B:uc}-\rf{B:3}
for possible analyitical bounds from above  and below on the support of such solutions in time. In particular, one could characterise a class of initial data with support less then $(0,1)$
with solution to \rf{SSM} having full support $(0,1)$ for all positive times. i.e. demonstrating the immediate initial healing of the dry spots.
\item Theorem 4.1 shows exponential decay to constant profile for
solutions to \rf{SSM} with $\sigma=0$ and special initial data \rf{CompCond},
see Remark 4.4. Numerical solutions of system \rf{SSM} indicate that the class of
initial data for which similar results hold can be extended (see Fig. 2,
second row). Also in the case $f\not=0$ in \rf{Lagr_eq_zcurv} the finite
decay should be then not to the constant profile $u=1$, 
but rather to the spatially inhomogeneous solution to stationary equation \rf{SS}.
In this context it is interesting to point out existence of the Lie-B\"acklund transform, which
maps equation \rf{Lagr_eq_zcurv} into the classical heat equation. Explicit connections between
solutions to these two equations were derived in~\cite{BK80}.
\item We should point out the difference of the results of Theorem 4.1 to known ones for equation \rf{Lagr_eq_zcurv}
considered on the whole real line $\mR$. In particular, there were shown in this case non-existence of $L^1$ solutions and
extinction phenomena ($u\go 0$) in the finite time (see~\cite{Va06} and references therein). In contrast, while considering \rf{Lagr_eq_zcurv} on the finite interval $(0,\,1)$ and
analysing convergence to $u=1$ instead, we were able to show existence of the unique strong positive solution to  \rf{Lagr_eq_zcurv}  and to escape from these singularities.
Also we believe that the exponential decay shown in estimate \rf{fct} cannot be improved to the finite time decay for the same reason. Indeed, linearising equation \rf{Lagr_eq_zcurv} at $u=1$
results in the classical heat equation for $\bar{u}=u-1$:
\bes
\bar{u}_t=\nu\bar{u}_{ss},
\ees
which suggests that close to $u=1$ solutions to  \rf{Lagr_eq_zcurv}  at most exponentially decay in infinite time.
\item The analytical argument of section 5 shows the right lower bound \rf{LB} for the minimum of the height in the regime $\nu\ll 1$, see Remark 5.1 and Fig. 1.
In order to make this argument rigorous we suggest that a comparison principle between solutions of \rf{int_eq} and the approximate equation \rf{appr_eq} should be proposed.
At the moment it is not obvious as both equations are nonlinear parabolic of the fourth-order.
\item Theorem A.2 in contrast to Theorem 2.1 relies on the uniform bound from below $m$ in \rf{h_mbnd} for the radially symmetric solutions to system \rf{SSMr}. In the spirit of this study it would be
interesting to provide an analytical argument, similar to one in section 5 for the one-dimensional system \rf{SSM}, showing that solutions to \rf{SSMr} can not rupture in finite time.
 Alternatively, one could think on possible analytical estimates for $m$ depending on the initial data $(h_0,v_0)$. Such results would have an important implication for possibility of point rupture in 2D viscous sheets considered in~\cite{VLW01}.
\end{itemize}

\section{Acknowledgements}
The authors would like to thank Jens Eggers for pointing to them out the low bound \rf{LB} for the solutions to \rf{SSM} as well as for valuable comments on the results of this study.
GK would like to acknowledge support from Leverhulme grant RPG-2014-226. GK gratefully acknowledges the hospitality of ICMAT during a research visit to Madrid.
Part of this research was performed during participation of authors at the thematic research program "Nonlinear Flows" of the Erwin Schr\"odinger International Institute for Mathematics and Physics.

\appendix
\section{Asymptotic decay for the radially symmetric solutions}
\renewcommand{\theequation}{A.\arabic{equation}}
\setcounter{equation}{0}

In 2D space the viscous sheet system \rf{SSM} has the form~\cite{VLW01}:
\begin{subequations}
\label{SSM2D}
\begin{align}
\mathbf{v}_t+(\mathbf{v}\cdot\nabla)\mathbf{v}&=\sigma\nabla\Delta h+\frac{\mu}{h}\nabla\cdot(h[\nabla\mathbf{v}+(\nabla\mathbf{v})^T+2(\nabla\cdot\mathbf{v})\mathbf{I}]),
\label{SSM2D1}\\
h_t&= -\mathrm{div}\left(h\mathbf{v}\right),
\label{SSM2D2}
\end{align}
\end{subequations}
where $\mathbf{v}\in\mR^2$ is the velocity in the plane $(x,y)\in B_1(0)$, where $B_1(0)$ is the unit disc and $\sigma,\,\mu$ denote again surface tension and viscosity. 

Here we are interested in the radial symmetric solutions to \rf{SSM2D} having the form:
\bes
h(x,y,t)=h(r,t),\quad \mathbf{v}(x,y,t)=v(r,t)\mathbf{e}_\mathbf{r}\quad\text{with}\ r=\sqrt{x^2+y^2}\in (0,\,1).
\ees
Under the last ansatz system \rf{SSM2D} reduces to~\cite{VLW01}:
\begin{subequations}
\label{SSMr}
\begin{align}
v_t+vv_r&=\sigma\Big(\frac{1}{r}[rh_r]_r\Big)_r+\frac{4\mu}{h}\left(\left[\frac{h}{r}(rv)_r\right]_r-\frac{vh_r}{2r}\right),
\label{SSM1r}\\
h_t&= -\frac{\left(rhv\right)_r}{r}.
\label{SSM2r}
\end{align}
\end{subequations}
We consider \rf{SSMr} in the domain $(0,\,T)\times(0,1)$ with
boundary conditions
\be
h_r(0,t)=h_r(1,t)=v(0,t)=v(1,t)=0
\lb{rBC}
\ee
which together with \rf{SSM2r} imply the mass conservation:
\be
\int_0^1h(r,t)r\,dr=\int_0^1h_0(r)r\,dr=M.
\lb{CMr}
\ee
We first establish the energy and entropy estimates for system \rf{SSMr}.
\begin{lemma}\label{A:Th1}
For classical positive solutions to problem \rf{SSMr}-\rf{rBC} the following energy and entropy inequalities hold:
\sbea
\tfrac{d}{dt}\int_0^1\left[\frac{hv^2}{2}+\sigma\frac{|h_r|^2}{2}\right]r\,dr&\leqslant&-2\mu\left[\int_0^1\frac{h|(vr)_r|^2}{r}\,dr\right],\nonumber\\
\lb{EEr}\\
\tfrac{d}{dt}\int_0^1\left[\frac{h}{2}\left(v+4\mu\frac{h_r}{h}\right)^2+\sigma\frac{|h_r|^2}{2}\right]r\,dr&\leqslant&-2\mu\sigma\int_0^1\frac{|(rh_r)_r|^2}{r}\,dr
+\frac{C}{\mu\sigma m^3}\int_0^1\frac{hv^2}{r}\,dr,\nonumber\\
\lb{EnEr}
\seea
provided
\be
h(r,t)\geqslant m>0\quad\text{for all}\ t\in[0,T).
\lb{h_mbnd}
\ee
\end{lemma}
\begin{proof}[Proof of Lemma~\ref{A:Th1}] Let us multiply equation \rf{SSM1r} by $hur$ and integrate it in space. After applying integration by parts several times and using equation \rf{SSM2r}
one obtains the equality
\be
\tfrac{d}{dt}\int_0^1\left[\frac{hv^2}{2}+\sigma\frac{|h_r|^2}{2}\right]r\,dr=-4\mu\left[\int_0^1h|v_r|^2r\,dr+\int_0^1\frac{hv^2}{r}\,dr+\int_0^1hvv_r\,dr\right].
\lb{Ms1_a}
\ee
Note that the last integral in \rf{Ms1_a} can be bounded by
\bes
-\int_0^1hvv_r\,dr\leqslant\tfrac{1}{2}\left[\int_0^1h|v_r|^2r\,dr+\int_0^1\frac{hv^2}{r}\,dr\right].
\ees
This together with \rf{Ms1_a} imply \rf{EEr}.

To show \rf{EnEr} we multiply equation \rf{SSM1r} by $rh_r$ and integrate it in space. The left-hand side of \rf{SSM1r} produces then:
\bea
&&\int_0^1(v_t+v_r)h_rr\,dr=\tfrac{d}{dt}\int_0^1vh_r r\,dr-\int_0^1vh_{rt}r\,dr+\int_0^1vv_xh_rr\,dr\nonumber\\
&&=\tfrac{d}{dt}\int_0^1vh_r r\,dr-\int_0^1(vr)_r\frac{(hvr)_r}{r}\,dr+\int_0^1vv_xh_rr\,dr\nonumber\\
&&=\tfrac{d}{dt}\int_0^1vh_r r\,dr-\int_0^1h|v_r|^2r\,dr-\int_0^1\frac{hv^2}{r}\,dr\nonumber\\\nonumber\\
&&=\tfrac{d}{dt}\int_0^1\left(vh_r+\frac{hv^2}{8\mu}+\frac{\sigma}{8\mu}|h_r|^2\right)r\,dr,
\lb{En_LHS}
\eea
where in the second equality above we used equation \rf{SSM2r}.
Next, for the curvature term at the right-hand side of  \rf{SSM1r} one obtains
\be
\int_0^1\Big(\frac{1}{r}[rh_r]_r\Big)_rh_rr\,dr=-\int_0^1\frac{1}{r}|[rh_r]_r|^2\,dr.
\lb{En_RHS_curv}
\ee
In turn, for the viscous term  at the right-hand side of  \rf{SSM1r} one obtains
\bea
&&\int_0^1\frac{rh_r}{h}\left(\left[\frac{h}{r}(rv)_r\right]_r-\frac{vh_r}{2r}\right)\,dr=-\int_0^1\left(\frac{rh_r}{h}\right)_r\frac{h}{r}(rv)_r\,dr-\tfrac{1}{2}\int_0^1\frac{v|h_r|^2}{h}\,dr\nonumber\\
&&=-\int_0^1\left(\frac{rh_r}{h}\right)_r\frac{(rhv)_r}{r}\,dr+\int_0^1\left(\frac{rh_r}{h}\right)_rh_rv\,dr-\tfrac{1}{2}\int_0^1\frac{v|h_r|^2}{h}\,dr\nonumber\\
&&=\int_0^1\frac{rh_r}{h}\left(\frac{(rhv)_r}{r}\right)_r\,dr+\int_0^1r\left(\frac{h_r}{h}\right)_rh_rv\,dr+\tfrac{1}{2}\int_0^1\frac{v|h_r|^2}{h}\,dr\nonumber\\
&&=-\int_0^1\frac{h_{tr}h_r}{h}r\,dr-\tfrac{1}{2}\int_0^1\frac{|h_r|^2}{h^2}(rhv)_r\,dr+\tfrac{1}{2}\int_0^1\frac{v|h_r|^2}{h}\,dr\nonumber\\
&&=-\int_0^1\frac{h_{tr}h_r}{h}r\,dr+\tfrac{1}{2}\int_0^1\frac{|h_r|^2}{h^2}h_t r\,dr+\tfrac{1}{2}\int_0^1\frac{v|h_r|^2}{h}\,dr\nonumber\\
&&=-\tfrac{1}{2}\tfrac{d}{dt}\int_0^1\frac{|h_r|^2}{h}r\,dr+\tfrac{1}{2}\int_0^1\frac{v|h_r|^2}{h}\,dr,
\lb{En_RHS_visc}
\eea
where in the fourth and fifths equalities above we have used \rf{SSM2r}
again. Next, we estimate the second integral at the right-hand side of
\rf{En_RHS_visc} using \rf{EEr} and \rf{h_mbnd} as
\bea
\int_0^1\frac{v|h_r|^2}{h}\,dr&\leqslant&\frac{1}{\sqrt{m^3}}\left(\int_0^1\frac{hv^2}{r}\,dr\right)^{1/2}\left(\int_0^1|h_r|^2r\,dr\right)^{1/2}||h_r||_\infty,
\nonumber\\
&\leqslant&\frac{C}{\sqrt{m^3}}\left(\int_0^1\frac{hv^2}{r}\,dr\right)^{1/2}\left(\int_0^1\frac{|(rh_r)_r|^2}{r}\,dr\right)^{1/2}\nonumber\\
&\leqslant&\frac{C}{8\mu\sigma m^3}\int_0^1\frac{hv^2}{r}\,dr+2\mu\sigma\int_0^1\frac{|(rh_r)_r|^2}{r}\,dr,
\lb{En_RHS_visc1}
\eea
where in the second inequality above we have used the estimate
\bes
||h_r||_\infty\leqslant\left(\tfrac{1}{2}\int_0^1\frac{|(rh_r)_r|^2}{r}\,dr\right)^{1/2}.
\ees
Finally, by combining estimates \rf{En_LHS}--\rf{En_RHS_visc1} one arrives at the entropy inequality \rf{EnEr}.
\end{proof}
Similar to the one-dimensional case (Theorem \ref{B:Th1}) the energy and entropy
inequalities \rf{EEr}-\rf{EnEr} allow us to show the global exponential
asymptotic decay, but now only for {\it positive} classical solutions to
\rf{SSMr} having globally lower bound \rf{h_mbnd}. We denote radial energy and entropy functionals as
\beas
S(u,\,h)&:=&\tfrac{1}{2}\int_0^1 \left\{h\left(v+4\mu\frac{h_r}{h}\right)^2+\sigma|h_r|^2\right\}r\,dr,\\
E(u,\,h)&:=&\tfrac{1}{2}\int_0^1\left[hv^2+\sigma|h_r|^2\right]r\,dr.
\eeas
\begin{theorem}[asymptotic exponential decay]\label{A:Th2}
Assume that initial data $(h_0,v_0) \in H^1(0,1) \times L^2(0,1)$,
\bes
\sigma>0\ \text{and}\ \mu>0
\ees
and 
\be
h(r,t)\geqslant m>0\quad\text{for all}\ t>0.
\lb{h_m_bound}
\ee
Then there exist positive constants $A_i,\,B_i,\,i=1,2$ depending only on $E(h_0,\,v_0),\,S(h_0,\,v_0)$ and parameters $\sigma,\,\mu$ and $m$ such that
\begin{equation}\label{A:2}
||h-M||_{H^1(B_1(0))} \leqslant A_1e^{- B_1\,t},\quad\text{and}\quad ||v||_{L^2(B_1(0))} \leqslant A_2e^{- B_2\,t}\quad\text{for all}\ t \geqslant 0.
\end{equation}
\end{theorem}
\begin{proof}[Proof of Theorem~\ref{A:Th2}]
Using the Poincar\'{e}  inequality
\bes
\int_{B_1(0)} |\nabla h|^2 \, dxdy\leqslant
\tfrac{1}{4}\int_{B_1(0)}|D^2h|^2\,dxdy,\quad \nabla h\cdot n|_{\partial B_1(0)}=0,
\ees
which reduces in the case of the radial-symmetric function $h(x,y,t)=h(r,t)$ to
\bes
\int_0^1 |h_r|^2r\, dr\leqslant \tfrac{1}{4}\int_0^1\frac{|(rh_r)_r|^2}{r}\,dr,
\ees
from the energy and entropy inequalities  \rf{EEr}--\rf{EnEr} we find that
\begin{equation}\label{AC:3}
\tfrac{\sigma}{2} \int_0^1 |h_r|^2r\,dr +\tfrac{\mu\sigma}{8}\int_0^T\int_0^1|h_r|^2r\,drdt\leqslant S(v_0,h_0)+\tfrac{C}{\mu\sigma m^3}.
\end{equation}
From (\ref{AC:3}),  $\|\nabla h\|_2^2$ is dominated by the solution of
$$
y'(t) = - B_1\, y(t), \ \ y(0) = A_1,
$$
where
$$
A_1=\tfrac{2}{\sigma}S(v_0,h_0)+\tfrac{2C}{\mu\sigma^2m^3},\
B_1=\tfrac{\mu}{4}.
$$
Solving for $y(t)$, we deduce that
$$
||\nabla h||_2^2 \leqslant y(t) = A_1\,e^{- B_1\,t} \mbox{ for all } t>0,
$$
which implies the first estimate in \rf{A:2}. The second estimate in \rf{A:2} follows by a similar argument using \rf{EEr} and \rf{h_m_bound}.
\end{proof}

\bibliographystyle{unsrtnat}
\bibliography{bibliography}
\clearpage
\addtocounter{tocdepth}{2}

\end{document}